\documentclass[a4paper, 12pt, reqno]{amsart}
\usepackage{amssymb, amstext, amscd, amsmath,verbatim,lineno, mathtools}
\usepackage{enumitem}
\usepackage[all,cmtip]{xy}
\usepackage{epigraph}
\usepackage{etoolbox}
\usepackage{comment}
\usepackage{latexsym}
\usepackage{amsthm}
\usepackage{geometry}
\usepackage{color}
\usepackage[all]{xy}
\usepackage{xcolor}

\usepackage[normalem]{ulem}

\usepackage{dsfont}

\usepackage[style= numeric, giveninits, maxbibnames=99, sortcites=true]{biblatex}
\DeclareFieldFormat{pages}{#1}
\renewbibmacro{in:}{}
\AtEveryBibitem{%
  \clearfield{issn} % Remove issn
  \clearfield{doi} % Remove doi
  \clearfield{isbn} %remove isbn

  \ifentrytype{online}{}{% Remove url except for @online
    \clearfield{url}
    }
    }
 \usepackage{hyperref}
 \hypersetup{colorlinks=true, linkcolor={blue}, citecolor={red}, urlcolor={black}}
\def\Z{\mathbb{Z}}

\def\V{\mathcal{V}}
\def\til{\tilde{1}}

\DeclareMathOperator{\dg}{deg}
\DeclareMathOperator{\ls}{span}

\def\1{1}

\usepackage{aliascnt}
\usepackage[capitalize, noabbrev, nameinlink]{cleveref}
\newtheorem{thm}{Theorem}[section]
\crefname{thm}{Theorem}{Theorems}

\newaliascnt{lem}{thm}
\newtheorem{lem}[lem]{Lemma}
\aliascntresetthe{lem}
\crefname{lem}{Lemma}{Lemmas}

\newaliascnt{prop}{thm}
\newtheorem{prop}[prop]{Proposition}
\aliascntresetthe{prop}
\crefname{prop}{Proposition}{Propositions}

\newaliascnt{cor}{thm}
\newtheorem{corollary}[cor]{Corollary}
\aliascntresetthe{cor}
\crefname{cor}{Corollary}{Corollary}

\newaliascnt{defn}{thm}
\theoremstyle{definition}
\newtheorem{defn}[defn]{Definition}
\aliascntresetthe{defn}
\crefname{defn}{Definition}{Definitions}

\newaliascnt{rmk}{thm}
\theoremstyle{remark}
\newtheorem{rmk}[rmk]{Remark}
\aliascntresetthe{rmk}
\crefname{rmk}{Remark}{Remarks}

\newtheorem{rmks}[thm]{Remarks}

\newaliascnt{example}{thm}
\newtheorem{example}[example]{Example}
\aliascntresetthe{example}
\crefname{example}{Example}{Examples}

\numberwithin{equation}{section}
\newcommand\pref[1]{(\ref{#1})} %-- without putting parentheses on labels

\def\source{\mathbf{s}}
\def\range{\mathbf{r}}
\def\twist{(\Sigma, i,q)}
\def\id{\epsilon}

\def\GADP{\text{gr-ADP}}
\def\GACP{\text{gr-ACP}}
\def\GAQP{\text{gr-AQP}}
\def\ADP{\text{ADP}}
\def\ACP{\text{ACP}}
\def\AQP{\text{AQP}}
\def\gtwist{(\Sigma, i,q,c_G)}
\def\gtwists{(\Sigma, G, \Gamma)}
\def\nor{N_{\star}(C)}
\def\sig{\Sigma_{\star}}
\def\G{G_{\star}}
\def\csig{c_{\sig}}
\def\csigsi{c_{\sig}^{-1}}
\def\cgs{c_{\G}}
\def\cgsi{c_{\G}^{-1}}
\def\cigma{c_{\Sigma}}
\def\csigi{c_{\Sigma}^{-1}}
\def\cg{c_G}
\def\cgi{c_G^{-1}}
\def\Gz{G_\id}
\def\Sz{\Sigma_\id}

\def\TSA{A_R(\G;\sig)}
\def\TS{A_R(G;\Sigma)}
\def\GTSA{A_R^{\mathrm{gr}}(G;\Sigma)}
\def\TSAU{A_R(G^{(0)};q^{-1}(G^{(0)}))}
\def\up{^{\uparrow}}
\DeclareMathOperator{\ID}{id}

\def\ind{i_1}
\def\indd{i_2}
\def\inddd{i_3}

\DeclareMathOperator{\spp}{supp}
\DeclareMathOperator{\spg}{supp_G}
\DeclareMathOperator{\Iso}{Iso}

\begin{document}

\title{On graded quasi-Cartan pairs and twisted Steinberg algebras}

\author[L.O. Clark]{Lisa Orloff Clark}
\address[L.O. Clark]{
School of Mathematics and Statistics\\
Victoria University of Wellington\\
PO Box 600\\
Wellington 6140 New Zealand}
\email{lisa.orloffclark@vuw.ac.nz}

\author[L.D. Naingue]{Lynnel D. Naingue}
\address[L.D. Naingue]{
Department of Mathematics and Statistics\\
College of Science and Mathematics\\
MSU-Iligan Institute of Technology\\
Iligan City, Philippines 9200} \email{lynnel.naingue@g.msuiit.edu.ph}

\author[J.P. Vilela]{Jocelyn P. Vilela}
\address[J.P. Vilela]{
Department of Mathematics and Statistics\\
College of Science and Mathematics\\
MSU-Iligan Institute of Technology\\
Iligan City, Philippines 9200} \email{jocelyn.vilela@g.msuiit.edu.ph}

\begin{abstract}
We generalise recent results about quasi-Cartan, Cartan and diagonal subalgebras by introducing graded versions. We show that there is a correspondence between \emph{graded algebraic quasi-Cartan/ Cartan/ diagonal} pairs and certain graded twisted Steinberg algebras and that the associated graded discrete twist is unique.   Our results include all discrete group algebras, and so are more general than the ungraded version. 
\end{abstract}

\thanks{This research is supported by Marsden Fund grant 21-VUW-156 from the Royal Society of New Zealand. It is part of the PhD thesis of the second-named author, who  is supported by the Philippine Department of Science and Technology - Science Education Institute (DOST-SEI) under the Research Enrichment (Sandwich) Program of the ASTHRDP.  The second-named author thanks the first-named author and Te Herenga Waka -- Victoria University of Wellington for their hospitality during her extended visit. Both the seond-named and third-named authors thank Premier Research Institute of Science and Mathematics (PRISM) of MSU-IIT. The authors also thank Nathan Brownlowe, Rizalyn Bongcawel and Lyster Rey Cabardo for helpful conversations and the anonymous referee for their suggestions that improved the paper.}


\maketitle

\section{Introduction}
In this paper, we show that there is a groupoid algebra model for any inclusion of a commutative subalgebra in a graded algebra satisfying ``Cartan-like" hypotheses.  That is, we show that a given graded algebra that contains a ``Cartan-like'' subalgebra is isomorphic to a groupoid algebra and the corresponding groupoid is unique.   
We are motivated by the results of the first author and many coauthors in the paper \cite{ACCCLMRSS} on diagonal, Cartan, and quasi-Cartan subalgebras. These subalgebra inclusions generalise how the diagonal subalgebra sits inside an algebra of square matrices, in order of increasing generality.  We generalise these results by adding a group graded structure.  In particular, if the grading is trivial, we recover the results of \cite{ACCCLMRSS}.  When the grading is nontrivial, our results are more general.  For example, given any discrete group algebra with the trivial inclusion of the subalgebra generated by the group identity, our results recover the group. The groupoid algebra model we use is called a twisted Steinberg algebra.  

The Steinberg algebra associated to an ample groupoid was introduced independently in \cite{Steinberg} and \cite{CFST} and was generalised to include groupoid extensions, that is, discrete twists, in \cite{TSA2022}.  A graded approach to (untwisted) Steinberg algebras first appeared in \cite{CFST} and has proven useful in a variety of settings, most notably in the groupoid recovery results of \cite{AHBS,BiceClark,CR, Steinberg2}.
 Here, we introduce \emph{graded discrete groupoid twists} and their associated graded twisted Steinberg algebras.    We also introduce \emph{graded algebraic diagonal/Cartan/quasi-Cartan pairs} and show that each is isomorphic to a graded twisted Steinberg algebra (\cref{gradediso}), and vice versa (\cref{iff}).
We then show that the groupoid twist associated to one of these graded pairs is unique (\cref{unique}), thus providing a new groupoid recovery result.

We take motivation from the C*-algebra paper \cite{mot2} where they add a graded structure to C*-algebraic diagonal and Cartan subalgebras in functional analysis \cite{Kumjian, Renault}.  Although our definitions and results parallel theirs, their proofs are highly analytic in key places, so we employ different algebraic techniques here.  Often, the techniques of \cite{ACCCLMRSS} suffice, but not always. 
For example, the results of both \cite{BCF2024} and \cite{ACCCLMRSS} heavily use the \emph{conditional expectation} to construct the algebra isomorphism. However, in the graded setting, the domain of the expectation is not the entire algebra, only the homogeneous component of the identity.  In the C*-algebra setting, the authors of \cite{mot2} are able to extend it to the entire algebra using abstract harmonic analysis.  We do some work to get around this, culminating in \cref{gradediso}.

Another significant place where the techniques of neither \cite{BCF2024} nor \cite{ACCCLMRSS} apply is in our groupoid recovery section. This is more subtle, but in the algebraic setting, it comes down to understanding which normalisers satisfy the ``local bisection hypothesis", see \cref{supportedbis}.

The paper is organised as follows. In \cref{prelim}, we begin with preliminaries for graded algebras, then the three types of (ungraded) algebraic pairs we study -- diagonal, Cartan and quasi-Cartan, and twisted Steinberg algebras. In \cref{gradedpairs}, we incorporate the concept of grading into these three algebraic pairs. Also, we define graded discrete $R$-twists and show that their associated twisted Steinberg algebras are graded algebras. We also show how suitable graded twisted Steinberg algebras give rise to these graded algebraic pairs. On the other hand, in \cref{homogrpds}, starting from a pair of graded algebras, we show how to construct a graded discrete twist.  We then show that its associated twisted Steinberg algebra, along with its diagonal subalgebra, is isomorphic to the pair of graded algebras we started with. In \cref{sectionrecover}, we show that our construction recovers the graded discrete $R$-twist. Finally, in \cref{sectionexample}, we demonstrate that our results are indeed more general than \cite{ACCCLMRSS} through an example.

\section{Preliminaries}\label{prelim}
\subsection{Commutative graded \texorpdfstring{$R$}{}-algebras without torsion}
Throughout this article: $R$ denotes a commutative unital ring; $R^{\times}$ is the group of units of $R$; $\Gamma$ is a group with identity $\id$; and $A$ is a \textit{$\Gamma$-graded $R$-algebra}, in the sense that there is a family $\{A_\gamma: \gamma \in \Gamma \}$  of additive $R$-submodules of $A$ such that 
\begin{enumerate}[label=(\roman*), font=\normalfont]
    \item $A=\bigoplus_{\gamma\in \Gamma} A_{\gamma}$, and
    \item $A_\gamma A_\lambda \subseteq A_{\gamma \lambda}$ for all $\gamma,\lambda \in \Gamma$, where $A_{\gamma}A_{\lambda}$ denotes the additive subgroup generated by all products $a_{\gamma}a^{'}_{\lambda}$ with $a_{\gamma} \in A_{\gamma}$ and $a^{'}_{\lambda} \in A_{\lambda}$.
\end{enumerate}
\noindent The set $\bigcup_{\gamma\in \Gamma} A_\gamma$ is called the set of \textit{homogeneous elements} of $A$. An additive submodule $A_\gamma$ is called the \emph{homogeneous component of degree $\gamma$}, or simply the \emph{$\gamma$-component} of $A$ and the nonzero elements of $A_\gamma$ are called \emph{homogeneous of degree $\gamma$}. We write $\deg(a)=\gamma$ if $a\in A_\gamma \setminus \{0\}$. Any nonzero element $a\in A$ has a unique expression as a finite sum between homogeneous elements, $a=\sum a_\gamma$ where $a_\gamma \in A_\gamma$. A \textit{graded homomorphism} of $\Gamma$-graded algebras is a homomorphism $f: A \to B$ such that $f(A_\gamma) \subseteq B_\gamma$.

Note that $A_\id$ is a subalgebra of $A$ and the identity $1_A \in A_{\id}$, if it exists. In what follows, we consider a commutative subalgebra $C\subseteq A_\id$ with idempotents $I(C)$ and we require $C$ to satisfy the \textit{without torsion} condition with respect to $R$ in the sense that 
\begin{equation}
    \text{if~} e \in I(C) \setminus \{0\} \text{~and~} t\in R \text{~satisfy~} te=0, \text{~then~} t=0. \tag{WT}\label{WT}
\end{equation}

Note that if $C$ has at least one nonzero idempotent, then \hyperref[WT]{Condition \pref{WT}} also implies that $R$ is \textit{indecomposable} in the sense that its only idempotents are $0$ and $1$. 
 Since we exclusively work with $R$-subalgebras $C$ that are generated by their idempotent elements, for us, $R$ is always indecomposable. 
If $C$ is spanned as an $R$-module by $I(C)$, then each element $c\in C$ can be written as $c=\sum_{i=1}^n t_ie_i$ with mutually orthogonal idempotents $e_1, \ldots, e_n \in I(C)$ and $t_1, \ldots, t_n \in R$, see \cite{ACCCLMRSS}.  The proof uses a sort of inclusion-exclusion argument. The proof of \cref{subhat}\ref{inj} uses a similar inclusion-exclusion argument. 

In general, for any commutative $R$-subalgebra $C$ in an $R$-algebra $A$, an element  $n\in A$ is a \emph{normaliser} of $C$ in $A$ if there exists $m\in A$ such that 
\begin{equation}\label{normaliser}
    mnm=m, ~ nmn=n, ~ \text{~and~} mCn \cup nCm \subseteq C .
\end{equation}
The collection of normalisers of $C$ in $A$ is denoted by $N(C)$ and if we need to specify the ambient algebra $A$, we use the notation $N_A(C)$.  
The set $I(C)$ forms a set of \emph{local units} for $A$ if for any $\{a_1, \ldots, a_k\} \subseteq A$, there exists $e\in I(C)$ with $ea_i=a_i=a_ie$ for all $i\in \{1,\ldots, k\}$. 
If $I(C)$ does form a set of local units for $A$, then $N(C)$ is an inverse semigroup and its set of idempotents is $I(C)$. Recall that an \emph{inverse semigroup} is a semigroup $S$ such that, for each $n\in S$, there is a unique element $n^{-1}$ (called the \emph{inverse} of $n$) satisfying $nn^{-1}n=n$ and $n^{-1}nn^{-1}=n^{-1}$. For any $n,m \in S$, $(nm)^{-1}=m^{-1} n^{-1}$ and $(n^{-1})^{-1}=n$. 
For each $n\in N(C)$, we write $n^{\dagger}$ for its inverse, which is the unique element $m \in N(C)$ satisfying \cref{normaliser}. A normaliser $n\in N(C)$ is called a \textit{free normaliser} if either $n \in C$ or $(n^\dagger n)(nn^\dagger)=0$ (this latter condition is equivalent to $n^2=0$), see \cite{ACCCLMRSS}.

\subsection{Algebraic Pairs}
We recall the definition of ungraded algebraic pairs introduced in \cite{ACCCLMRSS}. 

\begin{defn}\label{defn: algpairs}
    Let $A$ be an $R$-algebra, and let $C \subseteq A$ be a commutative subalgebra satisfying \hyperref[WT]{Condition \pref{WT}}.  We call a map $P:A \to C$ a \textit{conditional expectation} if $P$ is $R$-linear; $P|_{C} = \ID_C$; and $P(cac') = cP(a)c'$ for $a \in A$ and $c,c' \in C$.
In addition, suppose:
	\begin{enumerate}[label=(\roman*), font=\normalfont]
		\item \label{paircon1}The set $I(C)$ forms a set of local units for $A$;
		\item $C=\ls(I(C))$;
		\item $A=\ls(N(C))$;
		\item \label{paircon4} There exists a \textit{faithful} conditional expectation $P:A \to C$, in the sense that for every $a \in A\setminus \{0\}$, there exists $n \in N(C)$ such that $P(na) \neq 0$. 
	\end{enumerate}
	Then the pair $(A,C)$ is called an
	\begin{enumerate}[label=(\roman*), font=\normalfont]
		\item[(ADP)] \textit{algebraic diagonal pair} if $A$ is spanned by the free normalisers of $C$;
		\item[(ACP)] \textit{algebraic Cartan pair} if $C$ is a maximal commutative subalgebra of $A$;
		\item[(AQP)] \textit{algebraic quasi-Cartan pair} if there is a faithful conditional expectation $P:A \to C$ that is \textit{implemented by idempotents}, in the sense that for every $n\in N(C)$, there exists $e\in I(C)$ such that $P(n)=ne=en$.
	\end{enumerate}
\end{defn}

For convenience, we say the pair $(A,C)$ is an \ADP~ (resp. \ACP, \AQP) instead of algebraic diagonal (resp. Cartan, quasi-Cartan) pair. By \cite[Lemmas~3.5 and 3.6]{ACCCLMRSS}, every \ADP~ is an \ACP~ and every  \ACP~ is an \AQP.

\subsection{Groupoids}\label{section:groupoids}  A groupoid $G$ is like a group except the operation is only partially defined.  For a more detailed definition, see \cite{Sims}. We write $G^{(2)}$ for the set of composable pairs and write $\source$ and $\range$ for the functions from $G$ to $G$ such that
	\begin{enumerate}
		\item[(i)] $\source(\alpha)=\alpha^{-1}\alpha$,~~~~~ called the \textit{source} of $\alpha \in G$ and
		\item[(ii)] $\range(\alpha)=\alpha\alpha^{-1}$,~~~~~ called the \textit{range} of $\alpha \in G$.
	\end{enumerate}
The \textit{unit space} of $G$ is denoted $G^{(0)}:= \source(G)    = \range(G)$. For each $x\in G^{(0)}$, define $G^{x}:= \range^{-1}(x)$. For any two subsets $U,V$ of $G$, define
    \begin{equation*}
        UV:= \{\alpha\beta : (\alpha, \beta) \in (U \times V) \cap G^{(2)}\} ~~~~\text{and}~~~~ U^{-1}:= \{\alpha^{-1} : \alpha \in U\}.
    \end{equation*}

  A \textit{topological groupoid} is a groupoid endowed with a topology such that inversion and composition are continuous.  An open set $B \subseteq G$ in a topological groupoid is an \textit{open bisection} if $\range(B)$ and $\source(B)$ are open and $\range|_B$ and $\source|_B$ are homeomorphisms onto their images. The groupoid $G$ is \textit{\'{e}tale} if $\range$ (or equivalently, $\source$) is a local homeomorphism. If $G$ is \'{e}tale and Hausdorff,  then $G^{(0)}$ is open and closed in $G$.  Also, $G$ is \'{e}tale if and only if $G$ has a basis of open bisections. We say that $G$ is \textit{ample} if it has a basis of compact open bisections. For the majority of this paper, the groupoids we work with are ample Hausdorff groupoids. 

    The \textit{isotropy} of a groupoid $G$ is the set $\Iso(G):= \{ \alpha \in G: \range(\alpha)=\source(\alpha)\}$. We say that $G$ is \textit{principal} if $\Iso(G)=G^{(0)}$, and that $G$ is \textit{effective} if the topological interior of $\Iso(G)$ is equal to $G^{(0)}$.

    Given groupoids $G$ and $H$, we say that a map $\psi: G \to H$ is a \textit{groupoid homomorphism } if $(\psi \times \psi)(G^{(2)})\subseteq H^{(2)}$ and $\psi(\alpha\beta)=\psi(\alpha)\psi(\beta)$ for all $(\alpha,\beta)\in G^{(2)}$. In addition, if $\psi$ is a homeomorphism and $(\psi \times \psi)(G^{(2)}) \supseteq H^{(2)}$, then $\psi$ is called a \textit{groupoid isomorphism}.

    Let $\Gamma$ be a discrete group with identity $\id$ and $G$ be a topological groupoid. A \textit{$\Gamma$-grading} of $G$ is a continuous groupoid homomorphism $c: G \to \Gamma$;
such a function $c$ is often called a \textit{cocycle} on $G$. For each $\gamma\in \Gamma$, we write $G_\gamma := c^{-1}(\gamma)$. We say that $X \subseteq G$ is \emph{$\gamma$-graded} if $X\subseteq G_\gamma$, or simply say $X$ is \emph{homogeneous} if such $\gamma$ exists. Clearly, $G^{(0)}$ is $\id$-graded and if $X$ is $\gamma$-graded, then $X^{-1}$ is $\gamma^{-1}$-graded.  

    If $B$ and $D$ are homogeneous compact open bisections of an ample groupoid, then $B^{-1}$ and $BD$ are also homogeneous compact open bisections. Under these operations, the collection of homogeneous compact open bisections form an inverse semigroup (denoted by $\mathcal{S}_c^h$ in \cite[Page~418]{Steinberg2}, or by $S_{G,\star}$ in \cite[Page~1025]{AHBS}), and is in fact an inverse subsemigroup of the inverse semigroup of compact open bisections. 

\subsection{Discrete twists}
Let $G$ be an ample groupoid, and let $T$ be a subgroup of $R^{\times}$. A \textit{discrete twist} by $T$ over $G$ is a sequence 
		\[
	\xymatrix@1{
		G^{(0)} \times T \ar[r]^-{i} ~ &~ \Sigma \ar[r]^{q}  & G 
	}
	\]
where the groupoid $G^{(0)} \times T$ is regarded as a trivial group bundle with fibres $T$, $\Sigma$ is a groupoid with $\Sigma^{(0)}=i(G^{(0)} \times \{1\})$, and $i$ and $q$ are continuous groupoid homomorphisms that restrict to homeomorphisms of unit spaces, such that the following conditions hold:
\begin{enumerate}
    \item[(DT1)] The sequence is exact, in the sense that $i(\{x\} \times T)= q^{-1}(x)$ for every $x \in G^{(0)}$, $i$ is injective, and $q$ is surjective. 
    \item[(DT2)] \label{DT2}The groupoid $\Sigma$ is a locally trivial $G$-bundle, in the sense that for each $\alpha\in G$, there is an open bisection $B_\alpha$ of $G$ containing $\alpha$, and a continuous map $P_\alpha: B_\alpha \to \Sigma$ such that
    \begin{enumerate}[label=(\roman*), font=\normalfont] 
        \item \label{nD1}$q \circ P_\alpha = \operatorname{id}_{B_\alpha}$
        \item \label{nD2}the map $(\beta, t) \mapsto i(\range(\beta), t)P_\alpha(\beta)$ is a homeomorphism from $B_\alpha \times T$ to $q^{-1}(B_\alpha)$. 
    \end{enumerate}
    \item[(DT3)] The image of $i$ is \textit{central} in $\Sigma$, in the sense that $i(\range(\sigma), t)\sigma= \sigma i(\source(\sigma),t)$ for all $\sigma\in \Sigma$ and $t \in T$.
\end{enumerate}

We denote a discrete twist over $G$ by $(\Sigma,i,q)$, or by $\Sigma \to G$, or simply by $\Sigma$.  
We identify $\Sigma^{(0)}$ with $G^{(0)}$ via $i$ (or via $q|_{\Sigma^{(0)}}$). 
Recall from \cite[Page~15]{TSA2022} that there is a continuous free action of $T$ on $\Sigma$ given by 
\begin{equation*}
    t \cdot \sigma= i(\range(\sigma),t)\sigma \text{~for all $t\in T$ and $\sigma\in \Sigma$.}
\end{equation*}
When $T=R^{\times}$, we will say the $\twist$ is a \emph{discrete $R$-twist} over $G$, which is our main focus.   

 Note that \cite[Definition~4.1]{TSA2022}
requires $G$ and $\Sigma$ to be Hausdorff, but \cite[Corollary~2.3]{ACCCLMRSS} shows that $\Sigma$ is Hausdorff whenever $G$ is Hausdorff. Moreover, $\Sigma$ is ample whenever $G$ is ample. 
A continuous map $P_\alpha: B_\alpha \to \Sigma$ is called a (\textit{continuous}) \textit{local section} if it satisfies condition \hyperref[DT2]{(DT2)}. If $G$ is ample, then the open bisections from condition \hyperref[DT2]{(DT2)} can be chosen to be compact.

We say that two discrete twists $(\Sigma,i,q)$ and $(\Sigma',i',q')$ by $T \leq R^{\times}$ over Hausdorff ample groupoids $G$ and $G'$, respectively, are \textit{isomorphic} if there exist groupoid isomorphisms $\psi_{\Sigma}: \Sigma \to \Sigma'$ and $\psi_G:G \to G'$ such that the following diagram commutes:

    \begin{equation*}
        \xymatrix@1{
		G^{(0)} \times T \ar[r]^-{i}  \ar[d]_{\psi|_{G^{(0)}} \times \mathrm{id}}~ &~ \Sigma \ar[r]^{q} \ar[d]^{\psi_{\Sigma}}  & G \ar[d]^{\psi_G~~.} \\
		G'^{(0)} \times T  \ar[r]^-{i'} &\Sigma'  \ar[r]^-{q'} & G'}
    \end{equation*}

\subsection{Twisted Steinberg Algebras}
Given a topological space $X$ and a ring $R$, we write $C(X,R)$ for the $R$-module of locally constant functions from $X$ to $R$. We define the \emph{support} of $f \in C(X,R)$ to be the set $\spp(f):= \{ x\in X: f(x)\neq 0 \}$, which is a clopen set.

Let $\twist$ be a discrete $R$-twist over an ample Hausdorff groupoid $G$. A function $f\in C(\Sigma,R)$ is $R^{\times}$\textit{-contravariant} if $f(t \cdot \sigma)=t^{-1}f(\sigma)$ for all $t\in R^{\times}$ and $\sigma \in \Sigma$. Define
$$\TS:= \{ f \in C(\Sigma,R) : f \text{~is $R^{\times}$-contravariant and $q(\spp(f))$ is compact } \}.$$
For any $f\in \TS$, we write $\spg(f):= q(\spp(f)) \subseteq G$, which is clopen, and we have $\spp(f) = q^{-1}(\spg(f))$. Note that $\TS$ is an $R$-module with pointwise addition and scalar multiplication. Multiplication is a little more complicated, 
for $f, g \in \TS$, let $S: \spg(f) \to \spp(f)$ be any section (not necessarily continuous) for $q$ on $\spg(f)$. Then the following formula does not depend on the choice of $S$ and defines an associative multiplication on $\TS$, making it into an $R$-algebra, called the \emph{twisted Steinberg algebra} \cite{ACCCLMRSS}:  
\begin{equation*}
    (f*g)(\sigma) = \sum\limits_{\alpha\in G^{\range(\sigma)} \cap \spg(f)} f(S(\alpha))g(S(\alpha)^{-1}\sigma).
\end{equation*}

The following formula can be inferred from the proof of \cite[Lemma~2.7]{ACCCLMRSS}: for each compact open bisection $B \subseteq \Sigma$, we write $\til_B$ for the element of $\TS$ defined such that
\[
\til_B(\sigma) =\begin{cases}
    t^{-1}&\text{if } \sigma \in t\cdot B \text{ for some }t \in R^{\times}\\
    0&\text{otherwise.}
\end{cases}
\]
Notice that $\spg(\til_B) = q(B)$ and \cite[Proposition~2.8]{ACCCLMRSS} says that $\TS$ is spanned by the $\til_B$'s.

Parallel to the three types of algebraic pairs, there are three corresponding types of twists. Two of the types depend on the underlying groupoid, that is, if $G$ is principal, then $(A,C):= (\TS, \TSAU)$ is an \ADP, and if $G$ is effective, then $(A,C)$ is an \ACP, see \cite[Proposition~4.8]{ACCCLMRSS}. The third type is a condition on the associated twisted Steinberg algebra.  We say that a discrete $R$-twist $\twist$ over an ample Hausdorff groupoid $G$ \textit{satisfies the local bisection hypothesis} if for every normaliser $n$ of $\TSAU$ in $\TS$, the set $\spg(n)$ is an open bisection of $G$. If this is the case, then $(A,C)$ is an \AQP, see \cite[Lemma~4.2]{ACCCLMRSS}.

\section{Graded Algebraic Pairs from Graded Discrete Twists}\label{gradedpairs}
In this section, we introduce graded generalisations of the algebraic pairs and discrete twists, then show how to construct a graded algebraic pair from a graded discrete twist.

Suppose that $A=\bigoplus_{\gamma \in \Gamma}A_{\gamma}$ is a $\Gamma$-graded $R$-algebra, $C \subseteq A_\id$ is a commutative subalgebra such that $I(C)$ forms a set of local units for $A_\id$. Recall that an element $a \in A$ is called homogeneous if $a \in A_\gamma$ for some $\gamma \in \Gamma$.
An element $n \in A$ is a \emph{homogeneous normaliser} of $C$ if $n\in A_\gamma$ for some $\gamma \in \Gamma$ and there is a homogeneous $m \in A$ satisfying \cref{normaliser} for $n$.
We denote the set of homogeneous normalisers of $C$ by $\nor$. It follows from the definition that if $n \in \nor$, then $n \in N(C) \cap A_\gamma$ for some $\gamma \in \Gamma$.

The following observations are needed to establish that $\nor$ is an inverse semigroup.
\begin{lem}\label{I(C)}Let $A=\bigoplus_{\gamma \in \Gamma}A_{\gamma}$ be a $\Gamma$-graded $R$-algebra, let $C \subseteq A_\id$ be a commutative subalgebra such that $I(C)$ forms a set of local units for $A_\id$, and let $A=\ls(\nor)$. Then the following hold.
\begin{enumerate}[label=(\alph*), font=\normalfont]
   \item \label{norminv} Let $n\in N(C)$ with  $n \in A_\gamma$ for some $\gamma \in \Gamma$. Then a homogeneous $m\in A$ satisfying \cref{normaliser} for $n$, belong to $N(C) \cap A_{\gamma^{-1}}$. Consequently, $m\in \nor$ and $nm,mn \in A_\id$.
   \item \label{idem} Any idempotent $e \in I(C)$ is homogeneous. In particular, $e\in A_\id$.
    \item \label{localunit1} $I(C)$ forms a set of local units for $\nor$.
    \item \label{localunit2} $I(C)$ forms a set of local units for $A$.
\end{enumerate}
\end{lem}

\begin{proof}
For \ref{norminv}, it follows from the definition that $n \in N(C)$ so that $m\in N(C)$. Since $m$ is homogeneous, $m\in A_\lambda$ for some $\lambda \in \Gamma$, so we are left to show that $\lambda=\gamma^{-1}$. Note that $n=nmn$ implying that $\gamma=\gamma\lambda\gamma^{-1}$. Thus, $\lambda=\gamma^{-1}$, as needed. Consequently, $m\in \nor$ since $m\in A_{\gamma^{-1}}$ and $n$ is the homogeneous element satisfying \cref{normaliser} for $m$. Thus, $nm \in A_\gamma A_{\gamma^{-1}} \subseteq A_\id$ and $mn \in A_{\gamma^{-1}}A_\gamma \subseteq A_\id$.

Statement \ref{idem} follows from $I(C) \subseteq C \subseteq A_\id$.
For \ref{localunit1}, let $\{n_1, n_2, \ldots, n_k \} \subseteq \nor$. By definition, for each $n_i$, there exists homogeneous $m_i\in A$ satisfying \cref{normaliser} for the corresponding $n_i$. By \ref{norminv}, $m_in_i, n_im_i \in A_\id$ for each $i$. Since $I(C)$ forms a set of local units for $A_\id$, there exists $e \in I(C)$ such that for each $i$, $(m_in_i)e = m_in_i$ and $n_im_i = e(n_im_i)$. Thus, we have $n_ie = n_i = en_i$.

For \ref{localunit2}, let $\{a_1, \ldots, a_l\} \subseteq A$. Then for each $j= \{1, \ldots, l\}$, we have $a_j = \sum_{i=1}^{k} t_in_i$ where $t_i \in R$ and $n_i\in \nor$. By part \ref{localunit1}, there exists $e \in I(C)$ such that for each $i$, we have $n_ie = n_i = en_i$. Hence, we obtain $a_j=ea_j=a_je$.

\end{proof}

\begin{prop}\label{invsem}
Let $A=\bigoplus_{\gamma \in \Gamma}A_{\gamma}$ be a $\Gamma$-graded $R$-algebra, let $C \subseteq A_\id$ be a commutative subalgebra such that $I(C)$ forms a set of local units for $A_\id$, and let $A=\ls(\nor)$. Then the homogeneous normaliser $\nor$ is an inverse semigroup with inverse $n \mapsto n^{\dagger}$ and its set of idempotents is equal to $I(C)$. \end{prop}

\begin{proof}
    By \cref{I(C)}\ref{localunit2}, $I(C)$ forms a set of local units for $A$. Hence, by \cite[Lemma~2.15]{ACCCLMRSS} the set $N(C)$ is an inverse semigroup with inverse $k\mapsto k^{\dagger}$ (for $k \in N(C)$) and its set of idempotents is  equal to $I(C)$. Note that $n\in \nor$ implies $n\in N(C)$, so we first show that $n^{\dagger} \in \nor$. Since $n^{\dagger}$ is the unique element satisfying \cref{normaliser} for $n$. This implies $n^{\dagger}$ is homogeneous, and it follows from \cref{I(C)}\ref{norminv} that $n^\dagger\in \nor$ with $n^\dagger\in A_{\gamma^{-1}}$, as required.

    Next, we show that $\nor$ is closed, so let $n,m\in \nor \subseteq N(C)$. Then $n \in A_\gamma$ and $m \in A_\lambda$ for some $\gamma,\lambda \in \Gamma$ so that $nm\in A_{\gamma\lambda}$. Also, $n^{\dagger}\in A_{\gamma^{-1}}$ and $m\in A_{\lambda^{-1}}$. Since $N(C)$ is closed, $nm \in N(C)$ and $(nm)^{\dagger}$ satisfies \cref{normaliser} for $nm$. Moreover, $(nm)^{\dagger}$ is homogeneous since $(nm)^\dagger=m^\dagger n^\dagger \in A_{\lambda^{-`}}A_{\gamma^{-1}} \subseteq A_{\lambda^{-1}\gamma^{-1}}$. Thus, $nm\in \nor$, as needed.

    Finally, to see that $I(C)$ is the set of idempotents of $\nor$, we only need to show that any idempotent $e \in I(C)$ is homogeneous since $\nor \subseteq N(C)$. But this follow from \cref{I(C)}\ref{idem}, completing the proof. 
\end{proof}

There is a natural partial order on the inverse semigroup $\nor$ given by $n\leq m$ if and only if $n=nn^{\dagger}m$ (or, equivalently, if and only if $n=mn^{\dagger}n$), see \cite{Lawson}. The partial order is preserved by multiplication and inversion. The following are additional properties of $\leq$ on $\nor$. 

\begin{corollary}\label{res: partialproperties}
    Fix $t\in R^{\times}$. Let $n,m\in \nor$ and $e,f \in I(C)$. Then the following hold.
    \begin{enumerate}[label=(\alph*), font=\normalfont]
        \item $tn \in \nor$ with $(tn)^{\dagger}= t^{-1}n^{\dagger}$;
        \item if $n\leq m$, then $tn \leq tm$;
        \item if $0 \neq f \leq te$, then $t=1$; and
        \item $n-ne,n-en \in \nor$ with $(n-e)^{\dagger}=n^{\dagger}-en^\dagger$ and $(n-en)^\dagger = n^\dagger - n^\dagger e$ and $n-ne,n-en \leq n$.
    \end{enumerate}
\end{corollary}

\begin{proof}
    These properties follow from \cite[Lemmas~2.16~and~2.17]{ACCCLMRSS} since $\nor \subseteq N(C)$ and so we are left to show that $tn$ is homogeneous. But this holds since $A_\gamma$ is an $R$-submodule. 
\end{proof}

\begin{rmk}
    Note that in the above proof, $tn$ is homogeneous because $A_\gamma$ is an $R$-submodule, not only a subring. Indeed, we must require $A$ to be a graded algebra, not just a graded ring. We put emphasis on this since there is often a focus on graded rings (e.g \cite{AHBS, Steinberg2} for Steinberg algebras and \cite{NO2020} for Leavitt path algebras).
\end{rmk}

\begin{prop}
    \label{homoelement} Suppose that $a_\alpha\in A_\alpha$ for some $\alpha\in \Gamma$. Then there exist $n_i \in \nor \cap A_\alpha$ and $t_i \in R \setminus \{0\}$ such that $a_\alpha=\sum\limits_{i=1}^{k}t_in_i$.   
\end{prop}

\begin{proof}
    Since $A$ is spanned by $\nor$, there exists $n_i \in \nor$ and $t_i \in R \setminus \{0\}$ such that $a_\alpha=\sum\limits_{i=1}^{k}t_in_i$.  Notice that $t_in_i$ are homogeneous for each $i\in \{1, \ldots, k \}$. Since $a_\alpha$ is homogeneous of degree $\alpha$ and the decomposition is unique,  we can assume that $t_in_i=0$ for all $n_i \notin A_\alpha$. Moreover, since $t_i \neq 0$, each $n_i$ must be in $A_\alpha$.
\end{proof}

Considering the definition of a $\Gamma$-Cartan pair of $C^*$-algebras where $\Gamma$ is an abelian group (\cite[Definition~3.9]{mot2}), we define graded versions of the algebraic pairs where $\Gamma$ is not necessarily abelian. 

\begin{defn} \label{defgradedpair} Suppose that $A=\bigoplus_{\gamma \in \Gamma}A_{\gamma}$ is a $\Gamma$-graded $R$-algebra, and $C$ is a commutative subalgebra of $A_\id$. We say that the pair $(A,C)$ is a \textit{graded algebraic diagonal} (resp. \textit{Cartan}, \textit{ quasi-Cartan}) \textit{pair} if
	\begin{enumerate}[label=(\roman*), font=\normalfont]
		\item $(A_\id, C)$ is an ADP (resp. ACP, AQP); and
		\item $\nor$ spans $A$.
	\end{enumerate}
\end{defn}

For convenience, we say that the pair $(A,C)$ is a \GADP~ (resp. \GACP, \GAQP) instead of graded algebraic diagonal (resp. Cartan, quasi-Cartan) pair. Since $\ADP \implies \ACP \implies \AQP$, it follows that every graded algebraic diagonal pair is a graded algebraic Cartan pair and every graded algebraic Cartan pair is a graded algebraic quasi-Cartan pair.

\begin{example}
    Let $R$ be an indecomposable commutative ring. Consider the $\Z$-graded algebra $A:=M_n(R)$ with homogeneous components given by
$$A_s = \left\{ \sum\limits_{j-i=s} a_{ij}E_{ij} : a_{ij} \in R  \right\}, \text{~~~ for }1-n \leq s \leq n-1,$$
	and  $A_s=\{0\}$, otherwise. Then one can check that $(A,A_0)$ is a \GADP. 
\end{example}

If an $R$-algebra $A$ is trivially $\Gamma$-graded, in the sense that $A_\id=A$ and $A_\gamma=\{0\}$ for all $\gamma \neq \id$, then $(A,C)$ is a \GADP~(resp. \GACP, \GAQP) if and only if $(A,C)$ is an ADP (resp. ACP, AQP).

Now, considering the definition of classical twist graded by a discrete abelian group $\Gamma$ in \cite[Definition~2.9]{mot2}, we introduce graded structure on a discrete twist and $\Gamma$ is not necessarily abelian. 

\begin{defn}
    Let $G$ be an ample Hausdorff groupoid, and let $T$ be a subgroup of $R^\times$.
	A $\Gamma$\textit{-graded discrete twist} is a discrete twist 
	$(\Sigma,i,q)$ by $T$ over $G$ 
	together with continuous groupoid homomorphisms $\cigma: \Sigma \to \Gamma$ and $\cg: G \to \Gamma$ such that the following diagram commutes.
	
\begin{equation}\label{comdiag}
    	\xymatrix@1{
		G^{(0)} \times T \ar[r]^-{i} ~ &~ \Sigma \ar[r]^{q} \ar[dr]_{\cigma}  & G \ar[d]^{\cg~~.} \\
		&& \Gamma
	}
\end{equation}
\end{defn}
We denote a $\Gamma$-graded discrete twist over $G$ by $\gtwist$ or $\gtwists$, and by a \emph{graded discrete $R$-twist}, we consider $T=R^{\times}$.
Note that the maps $\cigma$ and $\cg$ are $\Gamma$-gradings on $\Sigma$ and $G$, respectively.

For each $\gamma \in \Gamma$, the sets $G_\gamma:=\cgi(\gamma)$ and $\Sigma_\gamma:= \csigi(\gamma)$ are clopen since $\cg$ and $\cigma$ are continuous and $\Gamma$ is discrete. 
The commutativity of \hyperref[comdiag]{Diagram \pref{comdiag}} 
implies $\Sz=q^{-1}(\Gz)$ and $\Gz^{(0)}=G^{(0)}$. Note that since $\Gz$ is an open subgroupoid, by restricting the maps $i$ and $q$ we get a discrete $R$-twist $\Sz$ over an ample Hausdorff groupoid $\Gz$ by \cite[Lemma~3.6]{BCF2024}.
The following result generalises \cite[Proposition~7.1]{TSA2022}.
\begin{prop}\label{gtsa}
    Let $G$ be an ample Hausdorff groupoid. Suppose that $\gtwists$ is a graded discrete $R$-twist. Then the twisted Steinberg algebra $\TS$ is a $\Gamma$-graded $R$-algebra with homogeneous components given by 
    \begin{equation}\label{gtsacomp}
        \TS_\gamma := \{ f \in \TS : \spg(f) \subseteq G_\gamma \}.
    \end{equation}
\end{prop}

\begin{proof}
    It is clear that $\TS_\gamma$ is an $R$-submodule of $\TS$ for each $\gamma\in \Gamma$. 
    The argument used in \cite[Lemma~3.1]{CS2015} can be used to show that for a compact open bisection $B\subseteq \Sigma$, the function $\til_B$ is of the form $\sum_{i=1}^n \til_{V_{\gamma_i}}$ for some $\{\gamma_1, \ldots, \gamma_n \} \subseteq \Gamma$. So to see that $\TS= \bigoplus_{\gamma\in \Gamma} \TS_\gamma$ it suffices to show that any finite collection $$\{f_i \in \TS_{\gamma_i}: 1 \leq i \leq n \text{~and each $\gamma_i$ is distinct from others}  \}$$
    is linearly independent. But this is clear because for any $i$, $\spp(f_i)=q^{-1}(\spg(f_i)) \subseteq \csigi(\gamma_i)$ so that  $\spp(f_i) \cap \spp(f_j) = \emptyset$ whenever $i\neq j$. 
    
    Finally, we show that if $f\in \TS_\gamma$ and $g\in \TS_\lambda$, then $f*g \in \TS_{\gamma\lambda}$. Take a section $S: \spg(f) \to \spp(f)$ and fix $\sigma \in \spp(f * g)$. By definition of the convolution, there is some $\alpha \in \spg(f) \subseteq G_\gamma$ such that $\range(\alpha) = \range(\sigma)$ and $f(S(\alpha))g(S(\alpha)^{-1}\sigma) \neq 0$. But $S(\alpha) \in \spp(f)$, implying that $S(\alpha)^{-1}\sigma \in \spp(g)$. Observe that $\alpha^{-1}q(\sigma)= q(S(\alpha)^{-1})q(\sigma)= q(S(\alpha)^{-1}\sigma) \in \spg(g) \subseteq G_\lambda$. Thus, $q(\sigma)=\alpha\alpha^{-1}q(\sigma) \in G_\gamma G_\lambda \subseteq G_{\gamma\lambda}$ implying that $\spg(f * g) \subseteq G_{\gamma\lambda}$, as needed. 
\end{proof}

Given a graded discrete $R$-twist $\gtwists$, we denote by $\GTSA$ the graded twisted Steinberg algebra with homogeneous components given by \cref{gtsacomp}. Observe that the subalgebra $\TS_\id$ contains $\TSAU$ since $G^{(0)} \subseteq G_\id$. 
Since $\Sigma$ is ample,  it admits a basis $\mathcal{B}$ of compact open bisections. If we replace $\mathcal{B}$ with the refinement $\mathcal{B}':= \{ B \cap \Sigma_\gamma : B \in \mathcal{B} \text{~and~} \gamma\in \Gamma \}$, we obtain a basis of homogeneous compact open bisections. With this observation, the proof of following lemma is routine and follows from the ungraded results \cite[Proposition~2.8~and~Corollary~2.10]{ACCCLMRSS}.

\begin{lem}\label{spannorm}
     Let $R$ be an indecomposable commutative ring, and let $G$ be an ample Hausdorff groupoid. Suppose that $\gtwists$ is a graded discrete $R$-twist, $A:= \GTSA$ and  $C:= \TSAU$. 
     \begin{enumerate}[label=(\alph*), font=\normalfont]
         \item \label{homofunction} For each $f\in \TS_\gamma$, there exists a finite set $\mathcal{F}$ of homogeneous compact open bisections of $\Sigma_\gamma$ with mutually disjoint images in $G$, and elements $r_B$ of $R$, for each $B\in \mathcal{F}$, such that $f=\sum_{B \in \mathcal{F}}r_B\til_B$. 
         \item \label{homonorm} For each homogeneous compact open bisection $B \subseteq \Sigma$, $\til_B$ is a homogeneous normaliser of $C$ in $A$. In particular, if $B \subseteq \Sigma_\gamma$ for some $\gamma\in \Gamma$, then $\til_B \in \TS_\gamma$.
     \end{enumerate}
\end{lem}

\begin{prop}\label{sathypo}
     Let $R$ be an indecomposable commutative ring, and let $G$ be an ample Hausdorff groupoid. Suppose that $\gtwists$ is a graded discrete $R$-twist, $A:= \GTSA$ and  $C:= \TSAU \subseteq A_\id$. The pair $(A_\id, C)$ satisfies conditions \ref{paircon1}-\ref{paircon4} of \cref{defn: algpairs} with faithful conditional expectation $P:A_\id \to C$ given by restriction of functions from $\Sigma_\id$ to $q^{-1}(G^{(0)})$. Moreover, $A$ is spanned by the homogeneous normalisers $\nor$ of $C$, that is, $A=\ls(\nor)$.
\end{prop}

\begin{proof} 
Recall that $\Sigma_\id \to G_\id$ is a discrete $R$-twist over an ample Hausdorff groupoid $G_\id$ with $G_\id^{(0)}=G^{(0)}$. We can identify $A_R(\Gz;\Sz)$ with $A_\id$ by extending functions by zero. The first statement follows from \cite[Lemma~4.2]{ACCCLMRSS}. The last statement follows from \cref{spannorm} and from the fact that $A= \bigoplus_{\gamma\in\Gamma} \TS_\gamma$ by \cref{gtsa}.
\end{proof}

The following corollary is a consequence of \cite[Proposition~4.8~and~Lemma~4.2]{ACCCLMRSS} and \cref{sathypo}. 
\begin{corollary}\label{iff}
    Suppose that $\gtwists$ is a graded discrete $R$-twist, $A:= \GTSA$ and  $C:= \TSAU$.  If $\Gz$ is effective, then $(A,C)$ is a \GACP ~and if $\Gz$ is principal, then $(A,C)$ is a \GADP. Moreover, the pair $(A,C)$ is a \GAQP ~if and only if  the twist $\Sz$ over $\Gz$ satisfies the local bisection hypothesis. 
\end{corollary}

\section{Graded discrete twists from graded algebraic pairs}\label{homogrpds}
Throughout this section, we assume that $A=\bigoplus_{\gamma \in \Gamma} A_\gamma$ is a $\Gamma$-graded $R$-algebra and $C\subseteq A_\id$ is a commutative subalgebra such that $C$ satisfies \hyperref[WT]{Condition \pref{WT}} and $I(C)$ forms a set of local units for $A_\id$ and $A=\ls(\nor)$.

A natural construction of a groupoid that arises from an inverse semigroup is the groupoid of ultrafilters, as described in \cite{ACCCLMRSS, Lenz, LMS2013}, and this uses the natural partial order. To construct the groupoid of ultrafilters on $\nor$, we introduce the following concepts. If $U\subseteq \nor$, then the \textit{upclosure} of $U$ is the set
	\[ U\up:= \{m\in \nor: \text{there exists $n\in U$ with $n\leq m$}\}.\] A \textit{filter} on $\nor$ is a subset $U \subseteq \nor\backslash \{0\}$ such that $U=U\up$ and whenever $m,n \in U$, there exists $p\in U$ such that $p\leq m,n$. An \textit{ultrafilter} is a maximal filter.

    The ultrafilters of $\nor$ are precisely the primitive elements of the inverse semigroup $\mathcal{L}(\nor)$ of all filters (including the trivial filter $\nor$), see \cite{LMS2013}. By \cite[Proposition~9.2.1]{Lawson}, the collection $\sig$ of all ultrafilters on $\nor$ forms a groupoid with the following structure.
\begin{enumerate}[label=(\arabic*), font=\normalfont]
    \item For each $U\in \sig$, $U^{-1}:= \{ n^{\dagger} : n \in U\}$.
    \item A pair $(U,V)$ of ultrafilters on $\nor$ is composable if and only if $mn\neq 0$ for all $m \in U$ and $n\in V$, and  $UV:= \{mn:m \in U, n\in V \}\up.$ Equivalently, a pair $(U,V)$ of ultrafilters on $\nor$ is composable if and only if $\source(U):=U^{-1}U$ is equal to $\range(V):= VV^{-1}$.
    \item The unit space of $\sig$ is given by $\sig^{(0)}= \{U \in \sig : U \cap I(C) \neq \emptyset\}$.
\end{enumerate}
For $n\in \nor$, we write \[\mathcal{V}_n:= \{ U \in \sig: n \in U \}.\] The collection $\{\mathcal{V}_n: n \in \nor \}$ forms a basis of open bisections for a topology on $\sig$ making it an \'{e}tale groupoid. For any $n,m \in \nor$, we have $\mathcal{V}_n\mathcal{V}_m=\mathcal{V}_{nm}$ and $\mathcal{V}_n^{-1}=\mathcal{V}_{n^{\dagger}}$ (see \cite[Lemma~3.2]{germs} for more details). By \cite[Lemma~2.22]{Lawson3} (see also \cite[Propositions~2.2 and 4.4]{germs}), $\sig^{(0)}$ is a Hausdorff subspace of $\sig$.

\newpage

\begin{rmks}\label{ultprop}~
\begin{enumerate}[label=(\arabic*), font=\normalfont]
     \item\label{c5.1} 
    There is a homeomorphism from the set of ultrafilters of $\nor$ that contain an element of $I(C)$ to the set of ultrafilters of the inverse subsemigroup $I(C)$, and so we often identify elements of $\sig^{(0)}$ with ultrafilters of $I(C)$ when convenient. See \cite[Proposition~2.2~and~Proposition~4.4]{germs} for details. Since $I(C)$ is a Boolean algebra (in the sense of \cite{LH2013}, which is called a ``generalised Boolean algebra'' in \cite{Lawson4, Steinberg}), Stone duality implies that $\V_e$ is compact for each $e \in I(C)$. We recall that for $e,f \in I(C)$, their join is $e \vee f = e+f-ef$, their meet is $e \wedge f= ef$, and their relative component is $e \backslash f=e-ef$.

    \item \label{ultprop1} We used the notation $\sig$ for the groupoid of ultrafilters on $\nor$ to distinguish it from the groupoid $\Sigma$ of ultrafilters on $N(C)$ in \cite[Section~5]{ACCCLMRSS}. Although $\nor \subseteq N(C)$, the groupoids $\sig$ and $\Sigma$ do not necessarily have a relationship since an ultrafilter on $\nor$ is not necessarily a filter on $N(C)$, except $\sig^{(0)}\cong\Sigma^{(0)}$ since elements of these can be identified with the ultrafilters of $I(C)$.  However, \cite[Lemmas~5.2~and~5.3]{ACCCLMRSS} hold if $\Sigma$ and $N(C)$ are replaced with $\sig$ and $\nor$, respectively. In fact, the proof follows similarly using \hyperref[ultprop]{Remark \ref{ultprop}}\ref{c5.1} and the additional properties of $\leq$ on $\nor$ in \cref{res: partialproperties}. In particular, for $t\in R^{\times}$ 
    if $U \in \sig$, then $tU=\{tm: m\in U \} \in \sig$. 
\end{enumerate}
\end{rmks}

Following the construction in \cite{ACCCLMRSS}, define a relation on $\sig$, the groupoid of ultrafilters on the inverse semigroup $\nor$, by 
	\begin{equation} \label{quotientaction}
	    U\simeq W \iff \text{there exists $t\in R^{\times}$ such that $U=tW$.}
	\end{equation}
It is straightforward to check that $\simeq$ is an equivalence relation. Define $\G$ as the quotient $\sig/{ \simeq}$ endowed with the quotient topology, and denote the corresponding quotient map by $q:\sig \to \G$.  We denote the equivalence class of $U\in \sig$ by $q(U)$. 

Using the techniques used in \cite[Section~5.2]{ACCCLMRSS}, one can check that the following hold:
\begin{enumerate}[label=(\arabic*), font=\normalfont]
    \item The quotient $\G:= \{q(U): U \in \sig\}$ is a groupoid with inversion given by $q(U)^{-1}=q(U^{-1})$, composable pairs $\G^{(2)}:=\{ (q(U), q(W)): (U,W) \in \sig^{(2)}   \}$, and composition given by $q(U)q(W)=q(UW)$ for all $(U,W) \in \sig^{(2)}$. We have $\source(q(U))=q(\source(U))$ and $\range(q(U))= q(\range(U))$ for all $U\in \sig$ and so $\G^{(0)}=q(\sig^{(0)})$. 
    \item The quotient map $q: \sig \to \G$ is a groupoid homomorphism and an open map that restricts to a homeomorphism of unit spaces. 
    \item The collections $\{\mathcal{V}_n: n \in \nor \}$ and $\{q(\mathcal{V}_n) : n \in \nor \}$ form a basis of compact open bisections for the topology on $\sig$ and $\G$, respectively. In particular, $\sig$ and $\G$ are ample groupoids.  
\end{enumerate}

We next establish that $\sig$ is graded by showing first that each ultrafilter in $\nor$ is contained in a homogeneous component of $A$.

\begin{prop}\label{degU} If $U$ is a filter (or an ultrafilter) on $\nor$, then there exists $\alpha \in \Gamma$ such that $\dg(n)=\alpha$ for all $n\in U$.
\end{prop}

\begin{proof}
    Let $n,m \in U$. Since $U$ is a filter on $\nor$, there exists $p \in U$ such that $p \leq n,m$. This implies that $np^{\dagger}p = p = mp^{\dagger}p$. Using \cref{I(C)}\ref{norminv}, we have
    $\dg(n) = \dg(n)\dg(p^{\dagger})\dg(p) = \dg(m)\dg(p^{\dagger})\dg(p) =\dg(m)$.
    Since $n,m \in U$ are arbitrary,  $\dg(n)=\alpha$ for all $n\in U$ and  for some $\alpha\in \Gamma$.  
\end{proof}

Define
\begin{equation*}
    \csig: \sig \to \Gamma \textnormal{~by~} \csig(U)=\dg(n), ~\textnormal{for any $U \in \sig$ and for any $n\in U$; and}
\end{equation*}
\begin{equation*}
     \cgs: \G \to \Gamma \textnormal{~by~} \cgs(q(U)) = \csig(U), ~\textnormal{for any~}q(U) \in \G.
\end{equation*}

Then $\csig$ is well-defined by \cref{degU} and $\cgs$ is well-defined using the following proposition.

\begin{prop}\label{degqU} If $U \in \sig$, then for any $V \in \sig$ such that $q(U)=q(V)$, we have $\csig(V) = \csig(U)$.
\end{prop}

\begin{proof}
    Let $V \in \sig$ such that $q(U)=q(V)$. Then $U \simeq V$ so that $U=tV$ for some $t \in R^{\times}$. If $u \in U$, then $u=tv$ for some $v\in V$. Since every homogeneous component of $A$ is an $R$-submodule, we have 
    $\csig(U) = \dg(u) = \dg(tv) = \dg(v) = \csig(V).$
\end{proof}

\begin{lem}\label{cG}
    The maps $\csig: \sig \to \Gamma$ and $\cgs: \G \to \Gamma$ are continuous groupoid homomorphisms with $\csig = \cgs \circ q$.
\end{lem}

\begin{proof}
    To show that $\cgs$ is continuous, fix $\alpha \in \Gamma$ and $q(U) \in \cgsi(\alpha)$. Take $m\in U$. By \cref{degU}, $\dg(m) = \csig(U) = \alpha$. Hence, $q(U) \in q(\mathcal{V}_m) \subseteq \csig^{-1}(\alpha)$ implying that $\cgs$ is continuous. Notice that $\csig = \cgs \circ q$ is a composition of continuous maps, and hence, $\csig$ is also continuous. Also, $\csig$ and $\cgs$ are groupoid homomorphisms from the groupoid structures of $\sig$ and $\G$. 
\end{proof}

It follows from \cref{cG} that the collections $\{\mathcal{V}_n: n \in \nor \}$ and $\{q(\mathcal{V}_n) : n \in \nor \}$ form a basis of homogeneous compact open bisections for the topology on $\sig$ and $\G$, respectively.

\begin{lem}\label{GHaus}
    If $(A,C)$ is a \GAQP, then $\G$ is Hausdorff. Moreover, $\sig$ is also Hausdorff. 
\end{lem}

Before we prove \cref{GHaus}, we first note that $(A_\id,C)$ is an \AQP~ and using the construction of groupoid of ultrafilters in \cite{ACCCLMRSS}, we obtain
    \begin{enumerate}[label=(\arabic*), font=\normalfont]
        \item the groupoid of ultrafilters on the normalisers $N_{A_{\id}}(C)$ of $C$ in $A_\id$, say $\Sigma_{A_\id}$,
        \item \label{it2:Ge}the groupoid $G_{A_\id} := \Sigma_{A_{\id}}/\simeq$ where $\simeq$ is the same equivalence relation on $\sig$ associated to $\nor$, but is restricted on $\csigsi(\id) $. 
    \end{enumerate}
By \cite[Theorem~5.6]{ACCCLMRSS}, $G_{A_\id}$ is Hausdorff. We show in the following proof that $G_{A_\id}$ is equal to $\cgsi(\id)$. 

\begin{proof}[Proof of \cref{GHaus}]
    It suffices to prove $\G^{(0)}$ is closed in $\G$ by \cite[Lemma~8.3.2]{Sims}. By \cref{cG}, $\cgs$ is continuous, so that $\cgsi(\id)$ is closed in $\G$. Thus, it is enough to show $\G^{(0)}$ is closed in $\cgsi(\id)$. We do this by showing $\cgsi(\id)= G_{A_\id}$, which is Hausdorff. 
    Observe that $N_{A_{\id}}(C) = \nor \cap A_{\id}$. Then, by \cref{degU}, $U \in \sig \cap A_{\id}$ if and only if $U \in \Sigma_{A_\id}$. 
    Now, if $q(U) \in G_{A_\id}$, then $\cgs(q(U)) = \id$. For the reverse containment,  if $q(U) \in \cgsi(\id)$, then $U \in \csigsi(\id)$ implying that $q(U) \in G_{A_\id}$. Thus, $\cgsi(\id)= G_{A_\id}$, as required. The last statement follows from \cite[Corollary~2.3]{ACCCLMRSS}. 
\end{proof}

The following theorem uses a similar argument of \cite[Theorem~5.6]{ACCCLMRSS} to show $\sig$ is discrete $R$-twist over $\G$, \cref{cG} shows that it is graded and given a \GAQP, $\G$ is Hausdorff by \cref{GHaus}, so that $\sig$ is also Hausdorff by \cite[Corollary~2.3]{ACCCLMRSS}.

\begin{thm}\label{grpdultra}
    Suppose that $A$ is a $\Gamma$-graded $R$-algebra, $C\subseteq A_\id$ is a commutative subalgebra such that $C$ satisfies \hyperref[WT]{Condition \pref{WT}}, $I(C)$ forms a set of local units for $A_\id$, $C=\ls(I(C))$ and $A=\ls(\nor)$. Let $\sig$ be the groupoid of ultrafilters on $\nor$, let $\G$ be the quotient of $\sig$ by the equivalence relation given by \cref{quotientaction}, and $q:\sig \to \G$ be the quotient map. Define $i:\G^{(0)} \times R^{\times} \to \sig$ by $i(q(U),t)=tU$ for $U \in \sig^{(0)}$ and $t\in R^{\times}$. Then $(\sig,\G,\Gamma)$ is a graded discrete $R$-twist where the following diagram commutes.
		\[
	\xymatrix@1{
        \G^{(0)} \times R^\times \ar[r]^-{i} ~ &~ \sig \ar[r]^{q} \ar[dr]_{\csig}  & \G \ar[d]^{\cgs} \\
		&& \Gamma
	}
	\] 
 If $(A,C)$ is a graded algebraic quasi-Cartan pair, then $\G$ is Hausdorff, and so is $\sig$. 
\end{thm}

\section{The Graded Isomorphism \texorpdfstring{$A \cong_{gr} \GTSA$}{}} \label{sectioniso}

Throughout this section, we assume that $(A,C)$ is a graded algebraic quasi-Cartan pair, $\sig$ is the groupoid of ultrafilters on $\nor$ and denote by $C(\sig, R)$ the $R$-module of continuous (or equivalently, locally constant) functions from $\sig$ to $R$ with pointwise operations.

The results and proofs in this section are similar to \cite[Section~6]{ACCCLMRSS} but extra care is needed since the groupoids we have are different from \cite{ACCCLMRSS} and the faithful conditional expectation we have is a mapping from $A_\id$ to $C$ and not from all of $A$, hence the reason for building the map first in $A_\id$ and then extending it to all of $A$. We use the map $\phi:C \to A_R(\sig^{(0)})$ that satisfies
\begin{equation}\label{phi}
    \phi(e)= \1_{\V_e} \text{~~for all $e \in I(C)$,}
\end{equation}
which is an isomorphism since $C$ satisfies \hyperref[WT]{Condition \pref{WT}}, see \cite[Th\'{e}or\`{e}me~1]{isophi}.

\begin{prop}\label{subhat}
    Suppose that $(A,C)$ is a \GAQP. Let $\sig$ and $\G$ be the groupoids constructed in the previous section. Let $\phi: C \to A_R(\sig^{(0)})$ be the isomorphism from \cref{phi}. For each $a_\alpha \in \bigcup_{\gamma\in \Gamma} A_\gamma$, there is a function $\widehat{a_\alpha }: \sig \to R$ defined by 
    \begin{equation*}
        \widehat{a_\alpha }(U) := \begin{cases}
        \phi(P(n^{\dagger}a_\alpha ))(\source(U)) & \text{if $\csig(U)= \alpha $ and $n\in U$,} \\
        0 & \text{otherwise}.
        \end{cases}
    \end{equation*}
Moreover, 
\begin{enumerate}[label=(\alph*), font=\normalfont]
    \item $\widehat{a_\alpha}$ is continuous;
    \item the map $a_\gamma \mapsto \widehat{a_\gamma}$ from $A_\gamma$ to $C(\sig, R)$ is $R$-linear for all $\gamma \in \Gamma$;
    \item \label{inj} the map $a_\alpha \mapsto \widehat{a_\alpha}$ from $\bigcup_{\gamma\in \Gamma} A_\gamma$ to $C(\sig, R)$ is injective;
    \item $\widehat{a_\alpha}(tU)=t^{-1}\widehat{a_\alpha}(U)$ for every $t\in R^{\times}$ and $U \in \sig$; and
    \item for $c\in C$, we have $\widehat{c}|_{\sig^{(0)}}= \phi(c)$, and $\spp(\widehat{c}) \subseteq i(\G^{(0)} \times R^{\times})$.
\end{enumerate}
\end{prop}

\begin{proof}
    Except for \ref{inj}, the same arguments used in the proof of \cite[Proposition~6.1]{ACCCLMRSS} can be applied in our setting, as discussed in \hyperref[ultprop]{Remark \ref{ultprop}}\ref{ultprop1}. We show \ref{inj}, which requires some additional work since we only have $P$ defined and faithful on $A_{\id}$.  
    We do this by showing that if $\widehat{a_\alpha}=0$, then $a_\alpha=0$. So for all $U \in \sig$, we have $\widehat{a_\alpha}(U)=0$. We claim that $\phi(P(na_{\alpha}))=0$ for all $n \in A_{\alpha^{-1}} \cap \nor$. Fix $n \in A_{\alpha^{-1}} \cap \nor$ and $U \in \Sigma^{(0)}$. We prove our claim by considering two cases. 
    First, suppose $nn^{\dagger} \in U$.  Then $U \in \source(\mathcal{V}_{n^{\dagger}})$ and hence there exists an ultrafilter $W \in \mathcal{V}_{n^{\dagger}}$ such that $\source(W)=U$. Thus we have
    \[
    \phi(P(na_{\alpha}))(U) = \phi(P(na_{\alpha}))(\source(W))=\widehat{a_\alpha}(W) =0.
    \]

    Now, suppose that $nn^{\dagger}\notin U$.
    Then $U\notin \V_{nn^{\dagger}}$, and since $P$ is a conditional expectation, observe that
    \[
    \phi(P(na_{\alpha}))(U) = \phi(nn^{\dagger})(U)\phi(P(na_{\alpha}))(U)=1_{\mathcal{V}_{nn^{\dagger}}}(U)\phi(P(na_{\alpha}))(U)=0,\]
as needed. 

    Here is where our argument deviates from \cite{ACCCLMRSS}. Let $m \in A_{\id} \cap \nor$ and $n \in A_{\alpha^{-1}} \cap \nor$. Then $mn \in A_{\alpha^{-1}} \cap \nor$ and by our claim, $\phi(P(mna_\alpha))=0$. Since $\phi$ is injective, $P(mna_\alpha)=0$. Moreover, since $P$ is faithful on $A_\id$, we deduce that $na_\alpha=0$ for all $n\in A_{\alpha^{-1}} \cap \nor$. 

    By \cref{homoelement}, we can write $a_\alpha= \sum_{i=1}^{l} t_in_i$ where $t_i \in R\setminus \{0\}$ and $n_i\in A_\alpha \cap \nor$ for each $i$. To prove that $a_\alpha =0$, we consider 
    \[ e:= \sum\limits_{\ind = 1}^{l} n_{\ind}n_{\ind}^{\dagger} - \sum\limits_{1 \leq \ind < \indd \leq l} n_{\ind}n_{\ind}^{\dagger} n_{\indd}n_{\indd}^{\dagger} - \ldots + (-1)^{l+1} n_{1}n_{1}^{\dagger}n_{2}n_{2}^{\dagger}\cdots n_{l}n_{l}^{\dagger}\]
    and show that $ 0= ea_\alpha = a_\alpha$. Observe that 
    \[ ea_\alpha = \sum\limits_{\ind = 1}^{l} n_{\ind}n_{\ind}^{\dagger}a_\alpha - \sum\limits_{1 \leq \ind < \indd \leq l} n_{\ind}n_{\ind}^{\dagger} n_{\indd}n_{\indd}^{\dagger}a_\alpha + \ldots + (-1)^{l+1} n_{1}n_{1}^{\dagger}n_{2}n_{2}^{\dagger}\cdots n_{l}n_{l}^{\dagger}a_\alpha  \]
and each term contains $n_i^{\dagger}a_\alpha$ for some $1 \leq i \leq l$, which is equal to $0$ since $n_i^{\dagger} \in A_{\alpha^{-1}} \cap \nor$. Hence, $ea_\alpha=0$. Now, to show that $ea_\alpha = a_\alpha$, we compute the first few terms. 
\small{
        \begin{align*} 
\left(\sum\limits_{\ind=1}^{l} n_{\ind}n_{\ind}^{\dagger}\right)a_\alpha
                &= \sum_{i,\ind=1}^{l} t_in_{\ind}n_{\ind}^{\dagger}n_i 
                =\sum\limits_{i=1}^{l}t_in_i + \sum\limits_{\substack{i,\ind=1\\ i\neq \ind}}^{l} t_i n_{\ind}n_{\ind}^{\dagger}n_i \\
\left(- \sum\limits_{1 \leq \ind < \indd \leq l} n_{\ind}n_{\ind}^{\dagger} n_{\indd}n_{\indd}^{\dagger}  \right) a_\alpha
               &= - \sum\limits_{\substack{i,\ind=1\\ i\neq \ind}}^{l} t_i n_{\ind}n_{\ind}^{\dagger}n_i - \sum\limits_{\substack{i, \ind, \indd=1 \\ i \neq \ind, \indd \\ \ind < \indd }}^{l} t_i n_{\ind}n_{\ind}^{\dagger} n_{\indd}n_{\indd}^{\dagger} n_i \\
\left( \sum\limits_{1 \leq \ind < \indd < \inddd \leq l} n_{\ind}n_{\ind}^{\dagger} n_{\indd}n_{\indd}^{\dagger}n_{\inddd}n_{\inddd}^{\dagger} \right) a_\alpha
                &= \sum\limits_{\substack{i, \ind, \indd=1 \\ i \neq \ind, \indd \\ \ind < \indd }}^{l} t_i n_{\ind}n_{\ind}^{\dagger} n_{\indd}n_{\indd}^{\dagger} n_i  + \sum\limits_{\substack{i, \ind, \indd=1 \\ i \neq \ind, \indd \\ \ind < \indd }}^{l} t_i n_{\ind}n_{\ind}^{\dagger} n_{\indd}n_{\indd}^{\dagger}n_{\inddd}n_{\inddd}^{\dagger} n_i.
    \end{align*}}
 Notice that terms cancel out so we are left with $ea_\alpha= \sum\limits_{i=1}^{l}t_in_i = a_\alpha$. But $ea_\alpha=0$ as shown previously, so $a_\alpha=0$, as required.     
\end{proof}

Since $A=\bigoplus_{\alpha\in \Gamma}A_\alpha$, we can extend the map we have in \cref{subhat} to all of $A$ as in the following corollary.
\begin{corollary}\label{hatmap}
    For each $a = \sum\limits_{\alpha\in \Gamma}a_\alpha \in A$ and $U \in \sig$, define a function $\widehat{a}:\sig\to R$ such that $$\widehat{a}(U)= \sum\limits_{\alpha\in \Gamma} \widehat{a_\alpha}(U).$$
    Then the following hold:
    \begin{enumerate}[label=(\alph*), font=\normalfont]
    \item  $\widehat{a}$ is continuous;
    \item \label{hatlinear} the map $a \mapsto \widehat{a}$ from $A$ to $C(\sig, R)$ is $R$-linear and injective;
    \item $\widehat{a}(tU)=t^{-1}\widehat{a}(U)$ for every $t\in R^{\times}$ and $U \in \sig$.
\end{enumerate}
\end{corollary}

For the proof of the main theorem (\cref{gradediso}) in this section, we use the following lemma. Using the notion of ultrafilters on $\nor$ and its properties, the proof of the following lemma is parallel to \cite[Lemma~6.3~and~Proposition~6.5]{ACCCLMRSS}. We again omit the details. 
\begin{lem}\label{c6}
Let $(A,C)$ be a \GAQP, and let $\sig$ and $\G$ be the groupoids constructed in \cref{homogrpds}. 
\begin{enumerate}[label=(\arabic*), font=\normalfont]
    \item \label{nhat} For $n\in \nor$ and $U \in \sig$, we have 
    \begin{equation*}
        \widehat{n}(U)= \begin{cases}
            t^{-1} & \text{if $U\in \V_{tn}$ for some $t\in R^{\times}$},\\
            0 & \text{otherwise}.
            
        \end{cases}
    \end{equation*}
    \item \label{surj}  For any $f \in \TSA$, there exist $n_1, \ldots, n_M \in \nor$ and $t_1, \ldots. t_M \in R$ such that $$f=\sum\limits_{j=1}^{M} t_j\widehat{n_j}.$$
\end{enumerate}
\end{lem}

\begin{thm}\label{gradediso}
    Suppose that $(A,C)$ is a \GAQP. Let $\sig$ and $\G$ be the groupoids constructed in \cref{homogrpds}. Then the map $\varphi: a \mapsto \widehat{a}$ from $A$ to $C(\sig,R)$ defined in \cref{hatmap} is a graded isomorphism of $A$ onto $\TSA$ that takes $C$ to $A_R(\G^{(0)}; q^{-1}(\G^{(0)}))$.
\end{thm} 

\begin{proof}
    We first show that $\widehat{a} \in \TSA$ for each $a\in A$. By \cref{c6}\ref{nhat}, if $n\in \nor$, then $\widehat{n} \in \TSA$. Since each $a\in A$ can be expressed as an $R$-linear combination of elements of $\nor$, we have $\widehat{a} \in \TSA$ because $\varphi$ is $R$-linear by \cref{hatmap}\ref{hatlinear}. \cref{hatmap}\ref{hatlinear} implies that $\varphi$ is an injective $R$-linear map, and \cref{c6}\ref{surj} implies that $\varphi$ is surjective. The same argument used in the proof of \cite[Theorem~6.6]{ACCCLMRSS} can be used to prove the isomorphism and the last statement. 
  
    To complete the proof that $\varphi$ is a graded isomorphism, we must show that $\varphi(A_\alpha) \subseteq \TSA_\alpha$. We do this by fixing $\alpha \in \Gamma$ and showing that $q(\spp(\widehat{a_\alpha})) \subseteq \cgsi(\alpha)$. Notice that $\spp(\widehat{a_{\alpha}})$ only contains ultrafilters $U$ on $\nor$ such that $\csig(U)=\alpha$. So, $q(\spp(\widehat{a_{\alpha}}))$ contains ultrafilters $q(U)$ in $\G$ such that $\alpha=\csig(U) = \cgs(U)$. Thus, $q(\spp(\widehat{a_\alpha})) \subseteq \cgsi(\alpha)$, completing the proof.
\end{proof}

The following corollary follows immediately from \cref{gradediso} and \cref*{iff}.

\begin{corollary}
    Let $(A,C)$ be a \GAQP. Then the discrete $R$-twist  $ \csigsi(\id) \to \cgsi(\id)$ of \cref{grpdultra} satisfies the local bisection hypothesis.
\end{corollary}

The following proposition describes the properties of the groupoid $\G$ that distinguish the graded algebraic diagonal and Cartan pairs amongst all graded algebraic quasi-Cartan pairs.  This is a consequence of \cite[Propositon~7.1]{ACCCLMRSS} and the fact that $\cgsi(\id)=G_{A_\id}=\Sigma_{A_\id} / \simeq$ , where $\Sigma_{A_\id}$ is the groupoid of ultrafilters constructed from the pair of algebras $(A_{\id},C)$.

\begin{prop}
    Suppose that $(A,C)$ is a \GAQP. Then 
    \begin{enumerate}[label=(\alph*), font=\normalfont]
        \item $(A,C)$ is a \GACP ~if and only if $\cgsi(\id)$ is effective, and
        \item $(A,C)$ is a \GADP ~if and only if $\cgsi(\id)$ is principal.
    \end{enumerate}
\end{prop}

%%%%%%%%%%%%%%%%%%%%%%%%%
\section{Recovering a graded discrete twist} \label{sectionrecover}

Throughout this section, let $\gtwists$ be a graded discrete $R$-twist, where $G$ is an ample Hausdorff groupoid, and let $(A,C):= (\GTSA, \TSAU)$. By applying \cref{grpdultra} to the pair $(A,C)$, we obtain a graded discrete $R$-twist $(\sig,\G,\Gamma)$. We will identify $\sig^{(0)}$ with ultrafilters on $I(C)$ via a homeomorphism described in \hyperref[ultprop]{Remark \ref{ultprop}}\ref{c5.1}.

Let $U$ be an ultrafilter on $I(C)$, and let $n\in \nor$ satisfy $n^{\dagger}n \in U$. The argument used in the proof of \cite[Lemma~8.1]{ACCCLMRSS} can be used to prove that the upclosure 
\[(nU)\up \in \mathcal{V}_n \subseteq \sig.
\]
For $\sigma\in \Sigma$, the set $U_{\source(\sigma)}:= \{e\in I(C): e(\source(\sigma))=1 \}$ is an ultrafilter on $I(C)$; to see this, note that $U_{\source(\sigma)}$ is nonempty since $\Sigma$ is ample; the map $B\mapsto \til_B$ is a lattice isomorphism between the compact open subsets of $\Sigma^{(0)}$ and $I(C)$ (see \cite[Corollary~2.12]{ACCCLMRSS}); and $I(C)$ is a Boolean algebra and every prime filter of a Boolean algebra is an ultrafilter (see \cite[Lemma~2.30]{Lawson3}). We are interested in the upclosure $(\til_XU_{\source(\sigma)})\up$ and for this to be in $\sig$, we need the normaliser $\til_X$ to be homogeneous, so we restrict to homogeneous compact open bisection $X \subseteq \Sigma$ (in the sense that $X \subseteq \Sigma_\gamma$ for some $\gamma \in \Gamma$, refer to \cref{spannorm}). 

The proofs for the next proposition and corollary use similar arguments from \cite[Proposition~8.4~and~Corollary~8.5]{ACCCLMRSS} using the appropriate notions of an upclosure and ultrafilter on the inverse semigroup $\nor$ instead of $N(C)$.  We show that the additional graded structure is also preserved. 
\begin{prop}\label{c8.4}
    Suppose that $\gtwists$ is a graded discrete $R$-twist, $A:= \GTSA$ and  $C:= \TSAU$. Let $\sig$ be the groupoid of ultrafilters on $\nor$ defined in \cref{homogrpds} and let $G_{\star}$ be the corresponding quotient of $\sig$ by the action of $R^{\times}$. For each $x \in \Sigma^{(0)}$, let $U_x$ be the ultrafilter $\{e\in I(C): e(x)=1\}$ on $I(C)$ corresponding to $x$. For $\sigma \in \Sigma$, let $\gamma= \cigma(\sigma)$
    , let $X \subseteq \Sigma_{\gamma}$ be a compact open bisection containing $\sigma$, and let 
    $$\Phi(\sigma):= (\til_X U_{\source(\sigma)})\up \in \sig.$$
    Then $\Phi$ does not depend on the choice of $X$. The map $\Phi: \Sigma \to \sig$ is an $R^\times$-equivariant continuous open graded  embedding of topological groupoids such that $\Phi(\Sigma^{(0)})= \sig^{(0)}$.
\end{prop}

\begin{proof}
   To see that the map $\Phi$ is graded, it suffices to show that $\cigma = \csig \circ \Phi$. Fix $\sigma \in \Sigma$. Let $\gamma=\cigma(\sigma)$, let $X\subseteq\Sigma_\gamma$ be a compact open bisection containing $\sigma$, and let $e \in U_{\source(\sigma)}$. By \cref{spannorm}\ref{homonorm}, $\til_X \in A_\gamma$ and by \cref{I(C)}\ref{idem}, $e \in A_\id$. Hence, $\til_{X}e\in A_\gamma$.
    Since $\til_Xe \in (\til_X U_{\source(\sigma)})\up  \in \sig$, observe that 
    \begin{equation*}
        \csig(\Phi(\sigma)) = \csig((\til_X U_{\source(\sigma)})\up) = \gamma = \cigma(\sigma), 
    \end{equation*}
    as required. 
\end{proof}

\begin{corollary} \label{reconstruction}
    Suppose that $\gtwists$ is a graded discrete $R$-twist, $A:= \GTSA$ and  $C:= \TSAU$. Let $(\sig,\G,\Gamma)$ be the graded discrete $R$-twist obtained from the pair $(A,C)$ via \cref{grpdultra}. Let $\Phi: \Sigma \to \sig$ be the map of \cref{c8.4}. Then there is a (well-defined) graded homomorphism $\Phi_G: G \to G_\star$ given by $\Phi_G(q(\sigma))= q'(\Phi(\sigma))$, and the following diagram commutes. 

    \begin{equation*}
        \xymatrixcolsep{3pc}\xymatrix@1{
		G^{(0)} \times R^{\times} \ar[r]^-{i}  \ar[dd]^-{\Phi_G \times \mathrm{id}}~ &~ \Sigma \ar[rr]^{q} \ar[dd]^{\Phi} \ar[rd]^{\cigma}  &~& G \ar[ld]_{\cg} \ar[dd]^-{\Phi_{G}~~} \\
        &&\Gamma \\
		G_\star^{(0)} \times R^{\times}  \ar[r]_-{i'} & \sig \ar[ru]_{\csig}  \ar[rr]_-{q'} & ~&G_\star \ar[lu]_{\cgs}}
    \end{equation*}
\end{corollary}

\begin{proof}
   We check that $\Phi_G$ preserves the grading, that is, $\cg = \cgs \circ \Phi_G$.
    Fix $\lambda \in G$. Since $q$ is surjective, there exists $\sigma \in \Sigma$ such that $q(\sigma)=\lambda$. Since $\gtwist$ is a graded discrete $R$-twist,  $\cigma(\sigma)= \cg(q(\sigma))= \cg(\lambda)$. Using this for the last equality and using \cref{cG} and \cref{c8.4} for the third and fourth equalities, respectively, we have
    \begin{equation*}
        \cgs(\Phi_G(\lambda)) = \cgs(\Phi_G(q(\sigma)) = \cgs(q'(\Phi(\sigma)) = \csig(\Phi(\sigma)) = \cigma(\sigma) = \cg(\lambda),
    \end{equation*}
    as needed.
\end{proof}

\begin{corollary}\label{suppbis}
    Let $n\in \nor$. Then $\spg(n)$ is a bisection if and only if $n=\til_X$ for a (necessarily unique) homogeneous compact open bisection $X$ of $\Sigma$. 
\end{corollary}
\begin{proof}
    This mostly follows from \cite[Proposition~8.6]{ACCCLMRSS} since $n\in N(C)$. We are left to show that such $X$ is contained in $\Sigma_\gamma$  for some $\gamma \in \Gamma$ whenever $n\in \TS_\gamma$. Since $\til_X=n\in \TS_\gamma$, we have $q(X) = \spg(\til_X) \subseteq G_\gamma$. Thus, $X \subseteq \Sigma_\gamma$, as required.  
\end{proof}

The following lemma shows that we only need the twist $\Sz$ to satisfy the local bisection hypothesis for $\spg(n)$ to be a bisection of $G$ for any homogeneous normaliser $n$.  

\begin{lem}\label{supportedbis}
Suppose that the twist $(\Sz, i',q')$ over $\Gz$ satisfies the local bisection hypothesis, and fix $n \in \nor$. Then $\spg(n)$ is a bisection of $G$.
\end{lem}

\begin{proof}
    We do this by showing $\range|_{\spg(n)}$ is injective. 
    Since $n$ is homogeneous, there exists $\gamma \in \Gamma$ such that $n\in A_\gamma$. Write $n=\sum_{B_i \in \mathcal{F}} r_{B_i}\til_{B_i}$ where $\mathcal{F}$ is a finite set of compact open bisections in $\Sigma_\gamma$ with mutually disjoint images in $G_\gamma$, and each $0\neq r_{B_i} \in R$. 
    Since
    \[\spg(n) = \bigcup_{\mathcal{F}}q(B_i)\]
    and $\range(q(B_i))=\range(B_i)$ for all $i$, to show $\range$ is injective on this union, it is enough to
    show $\range(B_i) \cap \range(B_j) = \emptyset$ for all $i \neq j$.
    We show $\range(B_0) \cap \range(B_1)=\emptyset$; a similar argument holds for any $i\neq j$.  
    By \cite[Corollary~2.10]{ACCCLMRSS},
    \[m:= \til_{B_0^{-1}} * n = r_{B_0}\til_{\source(B_0)}+ \sum_{i\neq 0}r_{B_i}\til_{B_0^{-1}B_i}.
    \]
    Then $m \in \nor \cap A_\id$ so $\spg(m)$ is a bisection in $\Gz$ by assumption. By way of contradiction, suppose there exist $\alpha\in B_0$ and $\beta \in B_1$ such that $\range(\alpha)=\range(\beta)$.  Since the $B_i$'s were chosen with mutually disjoint images in $G_{\gamma}$, we have $q(\alpha) \neq q(\beta)$. 
    
    We claim that both $\alpha^{-1}\alpha$ and  $\alpha^{-1}\beta$ are in $\spp(m)$. Suppose on the contrary that $\alpha^{-1}\alpha \notin \spp(m)$. Notice that $\alpha^{-1}\alpha\in \source(B_0)= \til^{-1}_{\source(B_0)}(1)$, so there must be an $i \neq 0$ such that $\alpha^{-1}\alpha \in \spp(\til_{B_0^{-1}B_i}) \subseteq R^{\times} \cdot B_0^{-1}B_i$. This means that $q(\alpha^{-1}\alpha) = q(\sigma\tau)$ for some $\sigma \in B_0^{-1}$ and $\tau \in B_i$. But $\range(q(\alpha^{-1}\alpha)) =\range(q(\sigma\tau))$ implies $\range(\alpha^{-1}\alpha)=\range(\sigma\tau)$, so that $\source(\alpha)=\range(\alpha^{-1})=\range(\sigma)=\source(\sigma^{-1})$. Since $\alpha,\sigma^{-1} \in B_0$ and $B_0$ is a bisection, $\alpha=\sigma^{-1}$. Thus, $q(\alpha)= q(\alpha)q(\alpha^{-1}\alpha)= q(\sigma^{-1})q(\sigma\tau)=q(\tau)$, implying that $B_0$ and $B_i$ do not have mutually disjoint images in $G_\gamma$, a contradiction. Hence, $\alpha^{-1}\alpha \in \spp(m)$, and a similar argument works for $\alpha^{-1}\beta\in \spp(m)$, proving our claim.   

    By our claim, $q(\alpha^{-1}\alpha), q(\alpha^{-1}\beta) \in \spg(m)$. Observe that $\range(q(\alpha^{-1}\alpha))= q(\range(\alpha^{-1}\alpha))= q(\range(\alpha^{-1}\beta))=\range(q(\alpha^{-1}\beta))$. But $\spg(m)$ is a bisection, so that $q(\alpha^{-1}\alpha)=q(\alpha^{-1}\beta)$. Thus, we get $q(\alpha)= q(\alpha)q(\alpha^{-1}\alpha)=q(\alpha)q(\alpha^{-1}\beta)= q(\beta)$, which is a contradiction. Thus, $\range|_{\spg(n)}$ is injective. A similar argument gives $\source|_{\spg(n)}$ is injective.
\end{proof}

Recall from \cref{section:groupoids} that the homogeneous compact open bisections of $\Sigma$ form an inverse subsemigroup. Also, recall from \cref{spannorm}\ref{homonorm} that $X\mapsto \til_X$ maps the homogeneous compact open bisections of $\Sigma$ to $\nor$. As in \cite[Proposition~8.2(c)]{ACCCLMRSS}, if $\til_X=\til_Y$, then $X=\til_X^{-1}(1)=\til^{-1}_Y(1)=Y$, so the map $X\mapsto \til_X$ is injective. For surjectivity: fix $n\in \nor$, assuming that $\Sigma_\id \to G_\id$ satisfies the local bisection hypothesis, \cref{supportedbis} implies $\spg(n)$ is a (compact open) bisection of $G$. Then, as in the proof of \cite[Proposition~8.6]{ACCCLMRSS}, $\til_{n^{-1}(1)}=n$, so the map is surjective. That the map is an isomorphism then follows from \cite[Corollary~2.10]{ACCCLMRSS}. Hence, we have the following proposition.

\begin{prop}
    Suppose that the twist $\Sigma_\id \to G_\id$ satisfies the local bisection hypothesis. 
    Then the map $X \mapsto \til_X$ is an isomorphism  between the inverse semigroups of homogeneous compact open bisections of $\Sigma$ onto $\nor$.
\end{prop}

Finally, we show in the next theorem that the map $\Phi:  \Sigma \to \sig$ of \cref{c8.4} is an isomorphism. This construction uses similar arguments used in \cite[Section~8]{ACCCLMRSS}. However, the surjectivity argument in \cite[Theorem~8.7]{ACCCLMRSS} requires the entire twist $\Sigma$ to satisfy the local bisection hypothesis, but \cref{supportedbis} tells us that we only need this condition to be satisfied on the twist $\Sz$. Indeed, if $\Sz$ satisfies the local bisection hypothesis, the isomorphism from homogeneous compact open bisections of $\Sigma$ to $\nor$ means $\Sigma$ is isomorphic to $\sig$ by graded noncommutative Stone duality \cite[Theorem~5.13]{GrLSduality}.

\begin{thm}\label{unique}
    Suppose that $\gtwists$ is a graded discrete $R$-twist, $A:= \GTSA$ and  $C:= \TSAU$. Let $\sig$ be the groupoid of ultrafilters on $\nor$ defined in \cref{homogrpds}. Let $\Phi: \Sigma \to \sig$ be the map from \cref{c8.4}. The following are equivalent.
    \begin{enumerate}[label=(\alph*), font=\normalfont]
        \item \label{8.7.1} The pair $(A,C)$ is a graded algebraic quasi-Cartan pair. 
        \item \label{8.7.2}  The twist $\Sz$ over $\Gz$ satisfies the local bisection hypothesis.
        \item \label{8.7.3}  The map $\Phi$ is surjective, and in particular, it is a graded isomorphism of topological groupoids. 
    \end{enumerate}
\end{thm}

\begin{proof}
    \cref{iff} gives \ref{8.7.1} $\iff$ \ref{8.7.2}. For \ref{8.7.2} $\implies$ \ref{8.7.3}, with \cref{supportedbis}, proof is similar to \cite[Theorem~8.7:(2)$\implies$(3)]{ACCCLMRSS}.  We provide the details. Fix $U \in \sig$ and $n \in U$.  Then $\spg(n)$ is a bisection by \cref{supportedbis} and hence there exists a compact open bisection $X \subseteq \Sigma$ such that $n=\til_X$ by \cref{suppbis}.  Since $\Phi$ is a bijection between the unit spaces by \cref{c8.4}, there exists $x \in \Sigma^{(0)}$ such that $\source(U)=(U_x)^{\uparrow}$. We claim that $x \in \source(X)$. Observe that $\til_{\source(X)}=\til_X^{\dagger}\til_X= n^\dagger n\in \source(U)= (U_x)\up$ since $n \in U$. Hence, by construction of $U_x$, there is some $e\in I(C)$ such that $e(x)=1$ and $e\leq \til_X$. Then there is some compact open subset $Y$ of $\Sigma$ such that $\til_Y = e \leq \til_{\source(X)}$ implying that $Y\subseteq \source(X)$ by \cite[Lemma~8.2(c)]{ACCCLMRSS}. Thus, $x\in e^{-1}(1)=\til_Y^{-1}(1) = Y \subseteq \source(X)$, proving our claim.  
    
    Now, let $\sigma$ be the unique element of $X$ with $\source(\sigma)=x$.  Then
    \[
    \Phi(\sigma)=(\til_XU_{\source(\sigma)})^{\uparrow}\subseteq (U\source(U))^{\uparrow}=U,
    \]
    and hence $\Phi(\sigma)=U$ because $U$ is an ultrafilter.  

    Finally, for \ref{8.7.3} $\implies$ \ref{8.7.2}, if $\Phi:\Sigma\to \sig$ is a graded isomorphism, it restricts to a groupoid isomorphism $\Phi:\Sigma_{\id} \to {\sig}_{\id}$.  Now we can apply  \cite[Theorem~8.7:(c)$\implies$(b)]{ACCCLMRSS} to get the result. 
\end{proof}

\begin{rmk}
    Our main result uses the homogeneous normalisers to establish a graded isomorphism to a graded twisted Steinberg algebra (\cref{gradediso}) and reconstruction result (\cref{unique}).  Thus, in the graded setting, a proper subset of normalisers is sufficient to get the kind of results in \cite{ACCCLMRSS} where *all* normalisers are used.  This is reminiscent of the recent work of the first-named author and coauthors in \cite{BCLM} for twisted groupoid C*-algebras. Further work to ``algebra-fy'' \cite{BCLM} is underway. 
\end{rmk}

%%%%%%%%%%%%%%%%%%%%%%%%%
\section{Examples}
\label{sectionexample}

It is shown in \cite[Section~9]{ACCCLMRSS} that in a discrete $R$-twist $\twist$ over a discrete group $G$ with identity $\id$, the local bisection hypothesis is equivalent to no nontrivial units condition for twisted group rings.  There are known examples of groups where this fails, even when $R$ is a field,  see for example \cite{Gardam}. In that case, by \cite[Lemma~4.2]{ACCCLMRSS}, the pair $(A,C):=(A_R(G,\Sigma), A_R(\{\id\}, q^{-1}(\{\id \}))$ is not an AQP. However, in the example below, we show that the pair $(A,C)$ has an obvious $\GAQP$ structure.   

\begin{example}
Let $R$ be a commutative unital ring and let $\twist$ be a discrete $R$-twist over a discrete group $G$ with identity $\id$.
    Consider the identity map $\cg: G\to G$ and define $\cigma := \cg \circ q$. Then $\gtwist$ is a graded discrete $R$-twist.  Let $A:= \GTSA$ and $C:= A_R(G^{(0)},q^{-1}(G^{(0)}))= A_R(\{\id\}, q^{-1}(\{\id \})$.  Note that we have \[
     C = A_\id  = A_R(\{\id\}, q^{-1}(\{\id \}))\cong R\{\delta_{\id}\} \cong R.
    \]
    This implies that all normalisers of $C$  in $A_{\id}$ are of the form $t\delta_\epsilon$ for some $t\in R$.  
    Since $\{\epsilon\}$ is a bisection, $\Sz$ satisfies the local bisection hypothesis. Therefore, by \cref{iff}, $(A,C)$ is a \GAQP.
\end{example}

\printbibliography

\end{document}